\documentclass{amsart}
\usepackage{amssymb, mathrsfs, amscd, amsmath, amscd, tikz}

\title[Specialization and divisibility for elliptic surfaces]{Specializations of elliptic surfaces, and divisibility in the Mordell-Weil group}
\author{Patrick Ingram}
\address{Department of Pure Mathematics, University of Waterloo}
\email{pingram@math.uwaterloo.ca}
\date{\today}

\newcommand{\QQ}{\mathbb{Q}}
\newcommand{\ZZ}{\mathbb{Z}}
\newcommand{\CC}{\mathbb{C}}
\newcommand{\RR}{\mathbb{R}}

\newcommand{\PP}{\mathbb{P}}

\newcommand{\Ocal}{\mathcal{O}}
\newcommand{\Es}{\mathcal{E}}

\newcommand{\Gal}{\operatorname{Gal}}
\newcommand{\Spec}{\operatorname{Spec}}
\newcommand{\GL}{\operatorname{GL}}
\newcommand{\SL}{\operatorname{SL}}
\newcommand{\ord}{\operatorname{ord}}
\newcommand{\MOD}[1]{~(\textup{mod}~#1)}
\newcommand{\h}{\hat{h}}
\newcommand{\im}{\operatorname{im}}

\newcommand{\pf}{\mathfrak{p}}

\newtheorem{theorem}{Theorem}
\newtheorem*{thm}{Theorem}

\newtheorem{lemma}[theorem]{Lemma}

\theoremstyle{remark}
\newtheorem*{remark}{Remark}
\newtheorem*{ack}{Acknowledgements}

\begin{document}

\begin{abstract}
Let $\Es\rightarrow C$ be an elliptic surface, defined over a number field $k$, let $P:C\rightarrow \Es$ be a section, and let $\ell$ be a rational prime.  
We bound the number of points of low algebraic degree in the $\ell$-division hull of $P$ at the fibre $\Es_t$.  Specifically, for 
$t\in C(\overline{k})$ with $[k(t):k]\leq B_1$ such that $\Es_t$ is non-singular, we obtain a bound on the number of $Q\in \Es_t(\overline{k})$ such that $[k(Q):k]\leq B_2$, and such that $\ell^nQ=P_t$, for some $n\geq 1$.  This bound depends on $\Es$, $P$, $\ell$, $B_1$, and $B_2$, but is independent of $t$.
\end{abstract}

\maketitle

\section{Introduction}

 One of the central problems in the study of elliptic surfaces is to determine the extent to which the geometry of the surface determines the arithmetic of its fibres.
Let $\Es\rightarrow C$ be an elliptic surface, defined over a number field $k$.  
Then if the fibre $\Es_t$ above $t\in C(k)$ is  non-singular,
there is a homomorphism 
\[\sigma_t:\Es(C)\rightarrow \Es_t(k),\]
where $\Es(C)$ is the group of sections $P:C\rightarrow \Es$ (we include the existence of a section in our definition of an elliptic surface).  By a theorem of Silverman \cite[Chapter III, Theorem~11.4]{jhs_advanced}, this map is injective for all but finitely many $t
\in C(k)$.  The map $\sigma_t$ is not, in general, surjective, since the rank of $\Es_t(k)$ may exceed that of $\Es(C)$ (see \cite{salgado}), but another result of Silverman shows that for elliptic surfaces over $\PP^1_\QQ$, there are infinitely many fibres for which the image of $\sigma_t$ is, at least,  not divisible in $\Es_t(\QQ)$.
\begin{thm}[Silverman \cite{jhs_specialization}]
Let $\Es\rightarrow\PP^1$ be an elliptic surface defined over $\QQ$, with non-constant $j$-invariant.  Then there exist infinitely many $t\in \PP^1(\QQ)$ such that the quotient  $\Es_t(\QQ)/\sigma_t(\Es(\PP^1))$ is torsion-free.
\end{thm}
Note that we cannot replace `there exist infinitely many' with `for all but finitely many', in Silverman's result, since the elliptic surface defined over $\PP^1_\QQ$ by $\Es:y^2=x^3-tx+t$ has no section of order two, but $(\eta, 0)\in \Es_t(\QQ)$ is a point of order two whenever $t=\eta^3/(\eta-1)$.

Saying that $\Es_t(\QQ)/\sigma_t(\Es(\PP^1))$ is torsion-free amounts to saying that for any prime $\ell$, and any section $P:\PP^1\rightarrow\Es$, the specialization $P_t=\sigma_t(P)$ is divisible by $\ell$ in the Mordell-Weil group $\Es_t(\QQ)$ only if $P$ is already divisible by $\ell$ in the group $\Es(\PP^1)$ of sections.  Since Silverman's result only treats infinitely many of the fibres, however,  it is still conceivable that one could construct a section $P:\PP^1\rightarrow\Es$ that is not divisible by a prime $\ell$, but whose specializations $P_t\in\Es_t(\QQ)$ are divisible by arbitrarily large powers of $\ell$ (as $t$ varies).

Our main result is that this sort of construction is not possible, and we prove this for elliptic surfaces $\Es\rightarrow C$ over arbitrary base curves, defined over a number field $k$.  In fact, even if we are allowed to consider fibres and points of bounded algebraic degree over $k$, the extent to which the specializations of a given section $P:\Es\rightarrow C$ might be $\ell$-divisible is limited.

\begin{theorem}\label{main}
Let $k/\QQ$ be a number field, let $\Es\rightarrow C$ be an elliptic surface, with non-constant $j$-invariant, over the smooth projective curve $C$, and let $P:C\rightarrow\Es$ be a section (all defined over $k$).  
 Then for any $B_1, B_2\geq 1$, there is a value $M(B_1, B_2)$ such that
$$\#\Big\{Q\in \Es_t(\overline{k}):[k(Q):k]\leq B_1\text{ and }\ell^nQ=P_t\text{ for some }n\geq 1\Big\}
\leq M,$$
as $t\in C(\overline{k})$ varies over the places of good reduction for $\Es$ with $[k(t):k]\leq B_2$.
\end{theorem}

In general, this finiteness result is the best we can do, since one is free to choose $P=\ell^NP_0$ for some section $P_0:C\rightarrow\Es$, and $N$ arbitrarily large, ensuring that the sets in the theorem have size at least $N$.  If, however, we cast out finitely many primes, and finitely many fibres, and consider only $k$-rational points, we
obtain something much more explicit.  We call the prime $\ell$ a \emph{special prime} for the elliptic surface $\Es$ if it is one of the finitely many primes such that either
 $\ell= 2$, or
 the $j$-invariant $j_\Es:C\rightarrow\PP^1$ has a pole of order divisible by $\ell$.
\begin{theorem}\label{special_th}
Let $\Es$ and $P$ be as above,  suppose that $\ell$ is not a special prime for $\Es$, and suppose further that $P$ is not of the form $\ell P_0$, for
any section $P_0:C\rightarrow \Es$.  Then
$$\#\Big\{Q\in \Es_t(k):\ell^n Q=P_t\text{ for some }n\geq 1\Big\}\leq \ell^2$$
for all but finitely many $t\in C(k)$.  Moreover, the upper bound can be replaced with $0$ if $C$ has genus at least 1, or if the $j$-invariant  $j_\Es$ has at least $5$ distinct poles in $C(\overline{k})$ (4 poles suffice if $\ell=5$, and 3 suffice if $\ell\geq 7$).
\end{theorem}

Of course, this theorem is trivially true if $C$ has genus 2 or greater.

\begin{remark}
The proof of Theorem~\ref{special_th} is a modification of the proof of Theorem~\ref{main}, and with slightly more work, one can obtain a version for
points of bounded degree over $k$.  Specifically, in proving the first claim of Theorem~\ref{special_th}, we actually prove  (under the conditions of the theorem) that $\ell^n Q=P_t$ implies $n\leq 1$, except on fibres corresponding to finitely many $t\in C(k)$.  If the argument is extended, one can show that for all but finitely many $t\in C(k)$, if  $Q\in \Es_t(\overline{k})$ with $[k(Q):k]\leq D$, then we have $\ell^n Q=P_t$ only if $n<\log_2 D+5$.  Note that, if we are allowed to consider points with $[k(Q):k]\leq D$, then we can find examples with $\ell^n Q=P_t$ for any $n\leq \log D/(2\log\ell)$.
\end{remark}

There is one important case in which Theorem~\ref{special_th} does not apply: if we'd like to discuss torsion on specializations of an elliptic surface, then we should like to apply Theorem~\ref{special_th} with $P=\Ocal$, the identity section.  However, it is always true that $\Ocal=\ell\Ocal$.  As it happens, this is not a fundamental obstacle.  For any finite set $S$ of rational primes, and any elliptic curve $E$, let $E^{\mathrm{Tors}, S}$ denote the $S$-primary torsion on $E$, that is, the union of $E[N]$, as $N\in\mathbb{N}$ ranges over $S$-units.
\begin{theorem}\label{torsion}
Let $\Es$ be as above, and suppose that $j_\Es$ has at least 5 distinct poles in $C(\overline{k})$.  Then for all but finitely many $t\in C(k)$, the torsion subgroup of $\Es_t(k)$ is exactly $\Es_t^{\mathrm{Tors}, S}(k)$, where $S$ is the set of special primes for $\Es$.
\end{theorem}

If $E/k$ is an elliptic curve over a number field, then one might consider $E$ as an arithmetic surface $E\rightarrow \Spec(R)$, where $R$ is the ring of
integers of $k$, and ask if the result
analogous to Theorem~\ref{main} holds.
Indeed, results of this general type, that is, local-to-global results about divisibility in the Mordell-Weil group, have already been considered over number fields, for example, the work of Banaszak, Gajda, and Kraso\'{n} \cite{gajda}.   However, for elliptic curves over number fields, it is quite easy to show that something rather different from Theorem~\ref{main} is true.

\begin{theorem}\label{cheb}
Let $k$ be a number field and let $E/k$ be an elliptic curve.  If $P\in E(k)$ is a point of infinite order, and $\ell$ is any rational prime, then for any $M\geq 1$, we may choose an infinite set of primes $S_M$ such that
$$\#\Big\{Q\in E_\pf(k_\pf):\ell^nQ=P_\pf\text{ for some }n\geq 1\Big\}\geq M$$
for all $\pf\in S_M$.  Moreover, we can choose our set of primes to have density $M^{-\frac{2}{5}+o(1)}$, where $o(1)\rightarrow 0$ as $M\rightarrow\infty$.
\end{theorem}
The results of Banaszak, Gajda, and Kraso\'{n} \cite{gajda} are largely Galois-theoretic, while those of Silverman \cite{jhs_specialization} are obtained by studying the variation of the N\'{e}ron-Tate height across fibres of $\Es$.  By contrast, Theorems~\ref{main} and \ref{special_th} require
require a mixture of Galois theory and some deep results in diophantine geometry.  Since it may be of independent interest, we mention here the main diophantine lemma used, which is a slight adjustment of Lemma 4.5 of \cite{preimpaper}, which in turn derives from work of Vojta \cite{vojta}, and Song and Tucker \cite{tucker}.

\begin{lemma}\label{preim}
Let
$$C_0\stackrel{\phi_1}{\longleftarrow}C_1\stackrel{\phi_2}{\longleftarrow}C_2\stackrel{\phi_3}{\longleftarrow}\cdots$$
be a tower of (smooth, projective) curves connected by non-constant morphisms, defined over the number field $k$.  Let $R_{\phi_n}$ denote the ramification divisor of $\phi_n$,
and suppose that there are constants $c_1> 0$ and $c_2$ such that
\[\frac{\deg R_{\phi_n}}{2\deg \phi_n}\geq c_12^n-c_2,\]
for all $n$.   Then for each $B\geq 1$, there exists an $N(B)$ such that $C_{N(B)}(\overline{k})$ contains at most finitely many points $Q$ with $[k(Q):k]\leq B$.
\end{lemma}

\begin{remark}
In fact, the proof (which is in Section~\ref{mainproof} below), shows that we may take $N(B)$ to be at most $\log_2 B+O(1)$, and $B\rightarrow\infty$.
\end{remark}


Before proceeding with the outline of the paper, we will make one remark about the requirement that $\Es$ have non-constant $j$-invariant.
While Silverman's result above has been considered in the case of elliptic surfaces $\Es\rightarrow\PP^1$ with constant $j$-invariant, by Gupta and Ramsay \cite{gupta}, it is clear that our main result cannot hold for all split elliptic surfaces.  If $E/k$ is an elliptic curve with rank at least 1, let $\Es\rightarrow E$ be an elliptic surface birational to $E\times E$, with projection onto the second coordinate.  Let $Q\in E(k)$ be a point of infinite order, and let $P$ be the diagonal section $P:E\rightarrow \Es$ defined by $t\mapsto (t, t)$.  Then for any $N$, if we set $t=\ell^NQ$, we clearly have
\[\#\left\{Q'\in E_t(k)=E(k):\ell^nQ'=P_t=\ell^NQ\right\}\geq N.\]

\begin{ack} I would like to thank David McKinnon and Xander Faber, for useful conversations during the writing up of these results, and Joseph Silverman and Soroosh Yazdani, for helpful comments on an earlier draft.
\end{ack}


\section{Notation and outline of the argument}\label{outline}

Although much of the proof of Theorem~\ref{main} will take place in function fields, it is useful to keep in mind the geometric picture.  The strategy of
the proof is as follows: let $\Es\rightarrow C$ be our elliptic surface, and let $\Gamma_0\subseteq \Es$ be the image of our section $P:C\rightarrow\Es$. 
For each $n$, let $\Gamma_{n+1}$ be the pull-back of $\Gamma_n$ by the rational function $[\ell]:\Es\rightarrow\Es$.  For any extension $F/k$, points in $\Gamma_n(F)$ parametrize fibres $\Es_t$ of $\Es$, with $t\in C(F)$, with a marked point $Q\in \Es_t(F)$ such that $\ell^n Q=P_t$.  In general, these curves might be singular and/or reducible, but we will imagine for the moment that their normalizations $\widetilde{\Gamma}_n$ are (geometrically) irreducible.  In other words, we have a tower
\[\widetilde{\Gamma}_0\longleftarrow\widetilde{\Gamma}_1\longleftarrow\widetilde{\Gamma}_2\longleftarrow\widetilde{\Gamma}_3\longleftarrow\cdots\]
of smooth projective curves, connected by dominant morphisms (namely, those induced by multiplication-by-$\ell$), all defined over the number field $k$.  Lemma~\ref{preim} gives us control over points of low algebraic degree on the curves $\widetilde{\Gamma}_n$, provided that the morphisms above ramify enough (equivalently, the genera of the curves increase quickly enough).
 Obtaining the appropriate lower bound on ramification provides for some tricky geometry, since the only possibility for ramification is where the curves $\Gamma_n$ intersect singular fibres of $\Es$, and these are precisely the points at which one might need to blow up in order the resolve the singularities of $\Gamma_n$.  We are saved by moving the entire problem into the function field setting, and applying Tate's non-archimedean uniformization of elliptic curves.  The resulting estimates on ramification, combined with Lemma~\ref{preim}, suffice to prove the results in this special case.

Generally, we can't hope for the curves $\widetilde{\Gamma}_n$ to actually be irreducible (in particular, if $P$ is a multiple by $\ell$ of another section, then $\widetilde{\Gamma}_1$ has a component birational to $C$), but each is the disjoint union of finitely many components, and the rational map $[\ell]:\Es\rightarrow\Es$ induces a map from each component of $\widetilde{\Gamma}_{n+1}$ to some component of  $\widetilde{\Gamma}_n$.  Denoting the components of the normalized curves by $\widetilde{\Gamma}_i^{(j)}$, we have a tree of curves with dominant morphisms which looks something like this:

\vspace{2mm}
\begin{center}
\begin{tikzpicture}
\draw (0,0) node[left] {$\widetilde{\Gamma}_0$};
\draw[<-] (0, 0.1) -- (1, .6);
\draw[<-] (0, -0.1) -- (1, -0.6);

\draw[<-] (2, .6) -- (3, .6);
\draw (2, .6) node[left] {$\widetilde{\Gamma}_1^{(1)}$};

\draw[<-] (2, -0.4) -- (3, -0.1);
\draw[<-] (2, -0.6) -- (3, -0.6);
\draw[<-] (2, -0.8) -- (3, -1.1);
\draw (2, -0.6) node[left] {$\widetilde{\Gamma}_1^{(2)}$};

\draw (4.5, .6) node[left] {$\widetilde{\Gamma}_2^{(1)} \cdots$};
\draw (4.5, -0.1) node[left] {$\widetilde{\Gamma}_2^{(2)} \cdots$};
\draw (4.5, -0.6) node[left] {$\widetilde{\Gamma}_2^{(3)} \cdots$};
\draw (4.5, -1.1) node[left] {$\widetilde{\Gamma}_2^{(4)} \cdots$};

\end{tikzpicture}
\end{center}
The key is to show that this tree  is eventually non-branching.  In other words, we want to show that the tree depicted above contains only finitely many infinite paths, so that we may apply Lemma~\ref{preim}  to each of these paths.
Looking at the surface as an elliptic curve $E$ over $K=k(C)$,  this amounts to showing there is some $N$ such that that the sets $[\ell]^{-n}P\subseteq
E(\overline{K})$ contain at most $N$ Galois orbits, any $n\geq 1$.    For elliptic curves over number fields, this follows from Kummer theory, but it
seems that these results have not previously been extended to elliptic curves over complex function fields.  In Section~\ref{galois}, we prove the
appropriate Galois-theoretic results to show that the number of components of the curves $\widetilde{\Gamma}_n$ eventually stabilizes.  In
Section~\ref{tate}, we will employ Tate's uniformization of elliptic curves over local fields to study the ramification of the maps
$\widetilde{\Gamma}_{n+1}\rightarrow\widetilde{\Gamma}_n$ induced by $[\ell]$.  In both of these sections, we consider $E$ over the extension $K\otimes_k
\CC$ of $K$, in order to obtain geometric results.  In Section~\ref{mainproof}, we assemble the proof of Theorem~\ref{main}, and in Sections~\ref{sharp} and
\ref{chebproof},
the proofs of Theorems~\ref{special_th} and \ref{cheb}.

Throughout the paper, $\Es$ is a smooth elliptic surface (with some chosen `identity section') defined over the number field $k$.  We denote by $E/K$ the generic fibre of $\Es$, an elliptic curve over the function field $K=k(C)$.  In order to obtain geometric results, in Sections~\ref{galois} and \ref{tate}, we will frequently work over the extension $K_\CC=K\otimes_k\CC$, determined by some fixed embedding of $k$ into $\CC$. The curves $\Gamma_n$ are as defined above, and $\widetilde{\Gamma}_n$ are their normalizations.  Since $\widetilde{\Gamma}_0\cong C$, we often identify these curves tacitly.  Also, in a slight abuse of notation, $P$ will stand both for the section $P:C\rightarrow\Es$, as well as the corresponding point in $E(K)$.


\section{Galois orbits}\label{galois}

For any prime $\ell$, the action of the absolute Galois group $\Gal(\overline{K_\CC}/K_\CC)$ on $E(K_\CC)$
partitions \[[\ell]^{-n}P=\{Q\in E(\overline{K_\CC}):\ell^nQ=P\}\] into a certain number of orbits for each $n$, and we wish to show that this number is bounded as $n\rightarrow\infty$.  In other words, we wish to show that Galois acts nearly as freely on $[\ell]^{-n}P$ as the group structure allows.  If $\ell$
is not a special prime, and $P
\not\in \ell E(K)$, we will show that $[\ell]^{-n}P$ is, in fact, Galois-irreducible, for all $n$ (in the sense that all of its elements are conjugate under the action of Galois).

For each $n$ we set $K_n=K_\CC(E[\ell^n])$, the $\ell^n$-division field of $E/K_\CC$, and we set $K_\infty$ to be the union of the $K_n/K_\CC$.  
Let
\[T_\ell(E)=\lim_{\rightarrow}E[\ell^n]\]
be the $\ell$-adic Tate module of $E$.  Fixing a basis for $T_\ell(E)$ allows us to define a representation
$$\rho_\ell:\Gal(K_\infty/K_\CC)\longrightarrow\GL_2(\ZZ_\ell).$$
In the number field case, one knows that such a representation is surjective for all but finitely many primes $\ell$ (unless the elliptic curve has
complex multiplication).  For elliptic curves over $K_\CC$, we cannot expect this to be true.  The Weil pairing $\mu$ on $E[\ell^n]$ sends pairs of torsion
points to roots of unity.  It is not hard to show that \[\mu(\sigma(T_1), \sigma(T_2))=\mu(T_1, T_2)^{\det(\rho_\ell(\sigma))},\] an
 in particular, since $K_\CC$
contains all roots of unity, we must have $\det(\rho_\ell(\sigma))=1$ for all $\sigma\in \Gal(K_\infty/K_\CC)$.
Thus, the image of the representation $\rho_\ell$ must be contained in $\SL_2(\ZZ_\ell)$.

\begin{lemma}\label{open_image}
Let \[\rho_\ell:\Gal(K_\infty/K_\CC)\longrightarrow\SL_2(\ZZ_\ell)\] be the $\ell$-adic Galois representation associated to $E/K_\CC$.  Then the image of
$\rho_\ell$ has finite index in $\SL_2(\ZZ_\ell)$.  Moreover, if $\ell$ is not a special prime for $\Es$, then $\rho_\ell$ is surjective.
\end{lemma}

\begin{proof}
The first claim is a theorem of Igusa \cite{igusa}.  More precisely, Cox and Parry \cite{cox-parry} show that the image of Galois in \[\SL_2(\widehat{\ZZ})=\prod_{\ell}\SL_2(\ZZ_\ell)\] contains the congruence
subgroup \[\Gamma(N)=\left\{M\in\SL_2(\ZZ_\ell):M\equiv\begin{pmatrix}1 & 0\\0 & 1\end{pmatrix}\MOD{N}\right\},\]
 for $N$  twice the least common multiple of the orders of the poles of the $j$-invariant $j_\Es:C\rightarrow \PP^1$.
  In other words, if $\ell$ is not a special prime for $\Es$ (i.e., if $
  \ell$ is odd, and prime to the orders of the poles of $j_\Es$), then the image of Galois is all of $\SL_2(\ZZ_\ell)$.
\end{proof}

We will need some facts about the Galois cohomology of elliptic curves over complex function fields.  In the number field case, the theory is reasonably well-understood, due to work of Bashmakov \cite{bashmakov} (see also Ribet \cite{ribet}).  The techniques rely, however, on the fact that in the number field setting, Galois acts on the $\ell$-primary torsion as an open subgroup of $\GL_2(\ZZ_\ell)$.  In particular, Bashmakov exploits elements of the centre of $\GL_2$, while the centre of $\SL_2$ is decidedly less interesting.  Nonetheless, the result we need is still true in this setting.

\begin{lemma}\label{cohom}
The first cohomology group  $H^1(\Gal(K_\infty/K_\CC), T_\ell(E))$ has finite exponent.  Furthermore, if $\ell$ is an odd prime, and if the representation $\rho_\ell$ is surjective, then the group  is trivial.
\end{lemma}

\begin{proof}
In the case where $\rho_\ell$ is surjective, the traditional proof works: the group $G=\Gal(K_\infty/K_\CC)\cong\SL_2(\ZZ_\ell)$ contains an element that acts as $-1$ on $T_\ell(E)$.    Since this element is in the centre of $G$, we know (by a lemma of Sah) that multiplication by $-2$ annihilates the first cohomology group.  Since the group $H^1(G, T_\ell(E))$ is $\ell$-torsion (multiplication by $\ell^m$ kills $E[\ell^m]$, and $H^1(G, T_\ell(E))$ is the projective limit of $H^1(\Gal(K_n/K_\CC), E[\ell^n])$),  it has exponent $\gcd(2, \ell)$.

We now treat the more general case, modifying an argument of Tate presented by Coates \cite{coates}.  For simplicity, choose a basis for $T_\ell(E)$, and identify $G$ with its image in $\SL_2(\ZZ_\ell)$.  By Lemma~\ref{open_image}, the image of the map
\[\Gal(K_\infty/K_\CC)\longrightarrow\SL_2(\ZZ_\ell)\] 
contains the congruence subgroup  $\Gamma(\ell^N)$, the kernel of reduction modulo $\ell^N$, for some $N\geq 0$.   Now, let $H \subseteq \Gamma(\ell^N)\subseteq G$ be the subgroup generated by the set of matrices
$$\left\{\begin{pmatrix}1 & 0\\\beta & 1\end{pmatrix}, \begin{pmatrix}1 & \beta \\ 0 & 1\end{pmatrix} :\beta\in \ell^N\ZZ_\ell\right\}.$$
We will show that $H^1(H, T_\ell(E))$ has finite exponent.  As noted at the beginning of the proof of Lemma~3.10 of \cite{bandini}, $H$ contains $\Gamma(\ell^{2N})$, and so 
\[(G:H)\leq (G:\Gamma(\ell^{2N}))\leq (\SL_2(\ZZ_\ell):\Gamma(\ell^{2N}))=\ell^{1+6N}(\ell^2-1).\] 
  Since the composition of the restriction and corestriction maps 
\[H^1(G, T_\ell(E))\stackrel{\mathrm{Res}}{\longrightarrow} H^1(H, T_\ell(E))\stackrel{\mathrm{Cor}}{\longrightarrow} H^1(G, T_\ell(E))\]
is simply multiplication by $(G:H)$, proving that $H^1(H, T_\ell(E))$ has finite exponent would be enough to prove the same of $H^1(G, T_\ell(E))$.

Let $f:H\rightarrow T_\ell(E)$ be a continuous 1-cocyle, that is, a map satisfying
$$f(\tau\sigma)=f(\tau)+\tau f(\sigma)$$
for all $\tau, \sigma\in H$.
For ease of reference, we will write
$$f\begin{pmatrix} 1 & \beta \\ 0 & 1\end{pmatrix}=\begin{pmatrix}f_1(\beta)\\f_2(\beta)\end{pmatrix}.$$
Note that, by the cocycle relation, 
\begin{eqnarray}
\begin{pmatrix}f_1(\beta_1+\beta_2)\\ f_2(\beta_1+\beta_2)\end{pmatrix}&=&f\begin{pmatrix}1 & \beta_1+\beta_2 \\ 0 & 1\end{pmatrix}\nonumber\\
&=&f\left(\begin{pmatrix} 1 & \beta_1 \\ 0 & 1\end{pmatrix}\begin{pmatrix} 1 & \beta_2\\ 0 & 1\end{pmatrix}\right)\nonumber\\
&=&\begin{pmatrix} f_1(\beta_1)\\f_2(\beta_1)\end{pmatrix}+\begin{pmatrix} 1 & \beta_1 \\ 0 & 1\end{pmatrix}\begin{pmatrix}f_1(\beta_2) \\ f_2(\beta_2)\end{pmatrix}\nonumber\\
&=&\begin{pmatrix} f_1(\beta_1)+f_2(\beta_2)+\beta_1f_2(\beta_2)\\f_2(\beta_1)+f_2(\beta_2)\end{pmatrix},\label{homomorphism}
\end{eqnarray}
and so, in particular, \[f_2(\beta_1+\beta_2)=f_2(\beta_1)+f_2(\beta_2)\] for all $\beta_1, \beta_2\in \ell^N\ZZ_\ell$.  Since cocycles send the identity
to the identity, $f_2$ is actually a homomorphism from $\ell^N\ZZ_\ell$ to $\ZZ_\ell$.

   Now, for any $\alpha\in 1+\ell^{2N}\ZZ_\ell$, we have $\begin{pmatrix} \alpha & 0 \\ 0 & \alpha^{-1}\end{pmatrix}\in H$.  This follows from the aforementioned comment in \cite{bandini}, or more directly from observing that for any $\gamma\in\ZZ_\ell$,
\[\begin{pmatrix} 1 & 0 \\ \frac{-\ell^N}{1+\ell^N\gamma} & 1\end{pmatrix}
\begin{pmatrix}1 & \ell^N\gamma \\ 0 & 1\end{pmatrix}
\begin{pmatrix}1 & 0 \\ \ell^N & 1\end{pmatrix}
\begin{pmatrix} 1 & \frac{-\ell^N\gamma}{1+\ell^N\gamma} \\ 0 & 1\end{pmatrix} =
\begin{pmatrix} 1+\ell^{2N}\gamma & 0 \\ 0 & (1+\ell^{2N}\gamma)^{-1}\end{pmatrix}.\]
  At this point, to simplify notation, we will write
\[\sigma_{\alpha, \beta}=\begin{pmatrix} \alpha & \beta\\ 0 & \alpha^{-1}\end{pmatrix}.\]
Now, on the one hand, we have the relation
\[\sigma_{\alpha, 0}\sigma_{1, \beta}\sigma_{\alpha, 0}^{-1}=\sigma_{1, \alpha^2\beta},\]
by simply multiplying the matrices.  On the other hand, it follows from the cocycle relation that 
$f(1)=f(1)+f(1)$, and so $f(1)=0$.  Furthermore,
\[0=f(1)=f(\sigma^{-1}\sigma)=f(\sigma^{-1})+\sigma^{-1}f(\sigma),\]
and so
$ f(\sigma^{-1})=-\sigma^{-1}f(\sigma)$ for all $\sigma$.  Thus, if $\alpha\in \ZZ\cap(1+\ell^{2N}\ZZ_\ell)$, (suppressing the first coordinate for convenience)
\begin{eqnarray*}
\begin{pmatrix}* \\ \alpha^2f_2(\beta) \end{pmatrix} =\begin{pmatrix}*\\ f_2(\alpha^2\beta)\end{pmatrix}& = & f(\sigma_{1, \alpha^2\beta})=f(\sigma_{\alpha, 0}\sigma_{1, \beta}\sigma_{\alpha, 0}^{-1})\\
&=& f(\sigma_{\alpha, 0})+\sigma_{\alpha, 0}f(\sigma_{1, \beta})+\sigma_{\alpha, 0}\sigma_{1, \beta}f(\sigma_{\alpha, 0}^{-1})\\
&=&f(\sigma_{\alpha, 0})+\sigma_{\alpha, 0}f(\sigma_{1, \beta})-\sigma_{\alpha, 0}\sigma_{1, \beta}\sigma_{\alpha, 0}^{-1}f(\sigma_{\alpha, 0})\\
&=&\begin{pmatrix}0 & -\alpha^2\beta\\ 0 & 0\end{pmatrix} f(\sigma_{\alpha, 0})+\sigma_{\alpha, 0}f(\sigma_{1, \beta})\\
&=&\begin{pmatrix} * \\ \alpha^{-1}f_2(\beta)\end{pmatrix}.
\end{eqnarray*}
Thus, for a given  $\beta$, we have $\alpha^2 f_2(\beta)=\alpha^{-1}f_2(\beta)$ for any integer $\alpha\equiv 1\MOD{\ell^{2N}}$.  Clearly, then, $f_2(\beta)=0$ for all $\beta$.

We now have $f_2$ vanishing identically, and, from \eqref{homomorphism} above, $f_1$ must be a homomorphism.  Again, we have $f_1(\alpha\beta)=\alpha f_1(\beta)$ for all $\alpha\in \ZZ\cap(1+\ell^{2N}\ZZ_\ell)$ and $\beta\in\ell^N\ZZ_\ell$.  Write
\[f\begin{pmatrix}\alpha & 0\\ 0 & \alpha^{-1}\end{pmatrix}=\begin{pmatrix} g_1(\alpha)\\ g_2(\alpha)\end{pmatrix},\]
for all $\alpha\in1+\ell^{2N}\ZZ_\ell$.   We compute again
\begin{eqnarray*}
\begin{pmatrix} \alpha^2f_1(\beta)\\ 0\end{pmatrix}&=& f(\sigma_{1, \alpha^2\beta})=f(\sigma_{\alpha, 0}\sigma_{1, \beta}\sigma_{\alpha, 0}^{-1})\\ & \vdots & \\
&=&\begin{pmatrix}0 & -\alpha^2\beta\\ 0 & 0\end{pmatrix}f(\sigma_{\alpha, 0})+\sigma_{\alpha, 0}f(\sigma_{1, \beta})\\
&=&\begin{pmatrix}-\alpha^2\beta g_2(\alpha)\\0\end{pmatrix}+\begin{pmatrix}\alpha f_1(\beta)\\0\end{pmatrix}.
\end{eqnarray*}
Thus, taking $\alpha=1+\ell^{2N}$, we obtain
\[\left(\alpha^2-\alpha\right)f_1(\beta)=-\alpha^2\beta g_2(\alpha),\]
and therefore 
\[\ell^{2N}f\begin{pmatrix}1 & \beta \\ 0 & 1\end{pmatrix}=\begin{pmatrix}-\alpha g_2(\alpha)\beta\\0\end{pmatrix},\]
for all $\beta$.  Using essentially the same argument, we can also show that
\[\ell^{2N}f\begin{pmatrix} 1 & 0 \\ \beta & 1\end{pmatrix}=\begin{pmatrix} 0\\ g_1(\alpha)\beta\end{pmatrix}.\]
Thus, for
\[\sigma\in\left\{\begin{pmatrix} 1 & \beta \\ 0 & 1\end{pmatrix}, \begin{pmatrix} 1 & 0 \\ \beta & 1\end{pmatrix}:\beta\in \ell^N\ZZ_\ell\right\},\]
we obtain
\[\ell^{2N}f(\sigma)=\sigma (\xi)-\xi\quad\text{for}\quad\xi=\begin{pmatrix}g_1(\alpha)\\ -\alpha g_2(\alpha)\end{pmatrix}.\]
The same relation must hold on all of $H$, since matrices of this form generate $H$, and so $\ell^{2N}f$ is a 1-coboundary.  Since $f$ was arbitrary,
$\ell^{2N}H^1(H, T_\ell(E))$ is trivial, proving the result.
\end{proof}

\begin{lemma}
For sufficiently large $m$, if  $Q\in E(\overline{K_\CC})$ and $\ell^m Q=P$, then $Q\not\in E(K_\infty)$.  If $\ell$ is not a special prime, and $P\not\in \ell E(K_\CC)$, then $\ell Q=P$ implies $Q\not\in E(K_\infty)$.
\end{lemma}

\begin{proof}
Suppose that $Q\in K_n$, and consider the map
\[f:G_n=\Gal(K_n/K_\CC)\rightarrow E[\ell^m]\]
defined by $f(\sigma)=\sigma(Q)-Q$.  Our first observation is that $f$ is a 1-cocyle.  This is simply because
\begin{eqnarray*}
f(\tau\sigma)&=&\tau\sigma(Q)-Q\\
&=&\tau\sigma(Q)-\tau(Q)+\tau(Q)-Q\\
&=&\tau f(\sigma)+f(\tau).
\end{eqnarray*}
By Lemma~\ref{cohom}, there is some $s$, not depending on $n$ and $m$, such that $\ell^sH^1(G_n, E[\ell^m])$ is trivial.  It follows that $\ell^sf$ is a 1-coboundary, so that $\ell^sf(\sigma)=\sigma(\xi)-\xi$, for some fixed $\xi\in E[\ell^m]$ and all $\sigma\in G_n$.  But then
\begin{eqnarray*}
\sigma(\ell^s Q-\xi)&=&\ell^s\sigma(Q)-\ell^sQ+\ell^sQ-\sigma(\xi)+\xi-\xi\\
&=&\ell^sf(\sigma)+\ell^sQ-\ell^sf(\sigma)-\xi=\ell^sQ-\xi,
\end{eqnarray*}
for all $\sigma\in G_n$, and so $\ell^sQ-\xi\in E(K_\CC)$. This implies \[\ell^sP=\ell^m(\ell^sQ-\xi)\in \ell^m E(K_\CC),\] which cannot be true if $m$
is large enough, since $s$ is  independent of $m$, and $E(K_\CC)$ is finitely generated.

If $\ell$ is not a special prime, let $\ell Q=P$ for some $Q\in E(K_n)$. The triviality of $H^1(G_n, E[\ell])$, by the argument above with $s=0$, leads to
$Q+\xi\in E(K_\CC)$, for some $\xi\in E[\ell]$.  This implies $P\in\ell E(K_\CC)$, which we have assumed is not true.
\end{proof}

Now, choose a consistent family of preimages of $P$, that is, a sequence $Q_s\in E(\overline{K_\CC})$ such that $Q_0=P$, and $\ell Q_{s+1}=Q_s$.  We consider the maps
\[f_s:\Gal(K_\infty(Q_s)/K_\infty)\longrightarrow E[\ell^s]\]
by
\[f_s(\sigma)=\sigma(Q_s)-Q_s.\]
Then $f_s$ is a homomorphism, since it is a 1-cocyle, and the domain acts trivially on the image.  It is also clear that $f_s$ is injective, since any $\sigma\in\Gal(K_\infty(Q_s)/K_\infty)$ fixing $Q_s$ must fix everything.  What's more, the diagram
\[\begin{CD}
\Gal(K_\infty(Q_{s+1})/K_\infty) @>\mathrm{restriction}>> \Gal(K_\infty(Q_s)/K_\infty)\\
@Vf_{s+1}VV @VVf_sV\\
\im(f_{s+1}) @>>[\ell]> \im(f_s)\\
\end{CD}\]
commutes (since $[\ell]$ is defined over the ground field), and so we may take (compatible) projective limits of the top and bottom.  If we let $K'$ denote the union of all fields $K_\infty(Q_s)$, we thereby obtain a map $f_\infty$ from $\Gal(K'/K_\infty)$ to the Tate module $T_\ell(E)$.
The group
\[H_\infty=\lim_{\rightarrow}\im(f_s)\subseteq \lim_{\rightarrow}E[\ell^s]=T_\ell(E)\]
is a submodule of $T_\ell(E)$ (both in terms of the Galois module structure, and the $\ZZ_\ell$ module structure).  To sum up, projective limits give us the following diagram, with exact rows
\[\begin{CD}
0 @>>>  \Gal(K'/K_\infty) @>>> \Gal(K'/K_\CC) @>>> \Gal(K_\infty/K_\CC) @>>> 0\\
& & @VVf_\infty V@VVV @VV\rho_\ell V & \\
0 @>>>  T_\ell(E) @>>> T_\ell(E)\rtimes\SL_2(\ZZ_\ell) @>>> \SL_2(\ZZ_\ell) @>>> 0.\\
\end{CD}\]
The next lemma shows that $\im(f)_\infty$ has finite index in $T_\ell(E)$.  Note that this is the same as showing that the image of $\Gal(K'/K_\CC)$ has finite index in $T_\ell(E)\rtimes\SL_2(\ZZ_\ell)$, in light of Lemma~\ref{open_image} and the exactness of the sequences above.

\begin{lemma}\label{hlemma}
The group $H_\infty$ has finite index in $T_\ell(E)$.  If $\ell$ is not a special prime, then in fact $H_\infty=T_\ell(E)$.
\end{lemma}

\begin{proof}
This simply follows from the fact that $H_\infty$ is a submodule of $T_\ell(E)$, and that the image of the action of Galois on $T_\ell(E)$ is an open subgroup of $\SL_2(\ZZ_\ell)$.  First suppose that $f_1$ is surjective, so that $\im(f_1)=E[\ell]$.  Then $H_\infty$ is a submodule of $T_\ell(E)$, with the property that $T_\ell(E)=H_\infty+\ell T_\ell(E)$.  It follows from Nakayama's Lemma (Lemma~4.2 on page 425 of \cite{lang}) that $H_\infty=T_\ell(E)$.

We now treat the general case.  First of all, it is clear that $H_\infty$ is not cyclic.  If $H_\infty$ were cyclic, then, as $H_\infty$ is a Galois submodule of $T_\ell(E)$, the image of the representation
\[\rho_\ell:\Gal(K_\infty/K_\CC)\longrightarrow \SL_2(\ZZ_\ell)\]
would be contained in a Borel subgroup, which clearly violates Lemma~\ref{open_image}.  So we may choose two linearly independent elements in $H_\infty$.  Let
\[\begin{pmatrix}u_1\ell^{a_1}\\u_2\ell^{a_2}\end{pmatrix}\qquad\text{and}\qquad\begin{pmatrix}v_1\ell^{b_1}\\v_2\ell^{b_2}\end{pmatrix}\]
be these two elements, with the $u_i$ and $v_i$ units in $\ZZ_\ell$.  By standard linear algebra, we may re-write this basis, and multiply by a power of $\ell$, to obtain $\ell^s e_1$, $\ell^s e_2\in H_\infty$, where $e_1$ and $e_2$ are the standard basis vectors.  But then $\ell^s T_\ell(E)\subseteq H_\infty$, and so
\[(T_\ell(E):H_\infty)\leq (T_\ell(E):\ell^sT_\ell(E))=\ell^{2s}.\]
This completes the proof. 
\end{proof}

We now state the main claim of this section, namely that Galois acts nearly as freely on $[\ell]^{-n}P$ as the group structure allows.

\begin{lemma}\label{orbits}
The number of distinct Galois orbits in $[\ell]^{-s}P$, over $K_\infty$, is bounded by \[(E[\ell^s]:\im(f_s))\leq (T_\ell(E):H_\infty).\]
\end{lemma}

\begin{proof}
Recall the point $Q_s$ such that $f_s(\sigma)=\sigma(Q_s)-Q_s$.  By definition, the Galois orbit of $Q_s$ is simply  $Q_s+\im(f_s)$.  Any other Galois orbit is of the form $Q'+\im(f_s)$, for some $\ell^s Q'=P$.  For each such $Q'$, there is a $\xi'\in E[\ell^s]$ such that $Q'=Q_s+\xi'$, and so the Galois orbit of $Q'$ is $Q_s+\xi'+\im(f_s)$.  This gives an explicit bijection between the Galois orbits of $\ell^{-s}P$, and the cosets of $\im(f_s)$ by $E[\ell^s]$.
That this number is bounded by the index $(T_\ell(E):H_\infty)$ follows from the fact that $T_\ell(E)/H_\infty$ is the projective limit of the groups $E[\ell^s]/\im(f_s)$.
\end{proof}


\section{Tate uniformization and ramification}\label{tate}

The aim of this section is to show that the tower of preimage curves described in Section~\ref{outline} is sufficiently ramified.  One can see, geometrically, why this must be true: if $v$ is a place of $C$ over which $E/K_\CC$ has split multiplicative reduction, the the fibre of the N\'{e}ron model over $v$ is the union of $v(j_E)$ lines (with intersection points removed).  The restriction of $[\ell]$ to any of the components of this N\'{e}ron polygon is an $\ell$-to-one map to some other component.
In particular, if the group of sections on $\Es$ contains all of the $\ell^n$th preimages of $P$, then $\ell^n\mid v(j_E)$.  This argument applies to extensions $\Es\times_C C'$ (given a N\'{e}ron model over the extension), as well, and so if $w$ is a prolongation of $v$ to a field over which $[\ell]^{-n}P$ is rational, for large $n$, we have $w(j_E)>v(j_E)$, so $v$ is ramified in this extension.

Although the geometric argument above can be turned into a proof, it is not entirely straightforward, in particular because the relationship between the the N\'{e}ron model of $\Es$ and the N\'{e}ron model of the base extension of $\Es\times_C C'$, for some curve $C'\rightarrow C$, is somewhat subtle when the covering is ramified.  We obtain a  simpler proof by considering the function field version of the problem. 
  The main tool is Tate's $v$-adic uniformization of elliptic curves, which is described over number fields in \cite[V.3-V.6]{jhs_advanced}.  The results over function fields are identical, and may be found in \cite{roquette}.
  
  Throughout this section, we fix a prime $v$ at which $E/K_\CC$ has split multiplicative reduction (we assume that one exists), and we suppose that $P$ does not reduce to the singular point modulo this prime.  In the proof of Theorem~\ref{main}, we will reduce the problem to the case where these assumptions hold.

\begin{thm}[Tate]
Let $F$ be a field, complete with respect to the non-archimedean valuation $v$, and suppose that $E/F$ is an elliptic curve with split multiplicative reduction at $v$.  Then there is a unique $q\in F^*$ with $|q|_v<1$, and maps such that
$$0\longrightarrow q^\ZZ\longrightarrow F^*\longrightarrow E(F)\longrightarrow  0$$
is an exact sequence.  Furthermore, if $F'/F$ is a Galois extension, then the corresponding sequence is an exact sequence of Galois modules.
\end{thm}

In essence, completing with respect to a prime of split multiplicative reduction, then, allows us to glean a lot of information about the elliptic curve $E$, by considering the multiplicative group of the completion of the field.  The following simple lemma will be used below.

\begin{lemma}\label{powers}
Let $F$ be field, complete with respect to the normalized discrete valuation $v$,  with ring of integers $R$, and an algebraically closed residue field $R/v$.
Then for any $\alpha\in F^*$, we have $\alpha\in (F^*)^n$ if and only if $n\mid v(\alpha)$.
 \end{lemma}

\begin{proof}
In one direction, note that if $\alpha\in (F^*)^n$, then $v(\alpha)=nv(\beta)$ for some $\beta\in F^*$.  It follows at once that $n\mid v(\alpha)$.

  Let $r_v:R\rightarrow R/v$ be the reduction-modulo-$v$ map.
First suppose that $\alpha\in R^*$.  Then $r_v(\alpha)\neq 0$, and (since the residue field is algebraically closed), $x^n-r_v(\alpha)$ has a simple root in $R/v$.  By Hensel's Lemma \cite[p.~34]{serre_local}, there is a root of $x^n-\alpha$ in $R$.  If $u$ is such a root, then $nv(u)=v(\alpha)=0$, and so $u\in R^*$, whence $\alpha=u^n\in (R^*)^n$.

Now suppose that $\alpha\in F^*\setminus R^*$.  By taking reciprocals if necessary, suppose that $\alpha\in R$. If $v(\alpha)=m$, write $\alpha=\gamma\pi^m$, where $\pi$ is a uniformizer for $v$, and $\gamma\in R^*$.  By the previous argument, $\gamma=u^m$ for some $u\in R^*$, and so $\alpha=(u\pi)^m\in(F^*)^m$.
\end{proof}

In order to state our next result, we will define the \emph{ramification tree} of the point $Q_0=P_v$ on $\Gamma_0\subseteq\Es$.  The nodes of the tree are the points on the curves $\widetilde{\Gamma}_n$ which map down to $Q_0$, with a point $Q_{n+1}$ on $\widetilde{\Gamma}_{n+1}$ linked to a point $Q_n$ on $\widetilde{\Gamma}_n$ if $Q_{n+1}$ maps to $Q_n$ by the map induced by $[\ell]$.  We will weight these edges with the ramification index of this map at $Q_{n+1}$, so that the weights of the edges above any given point sum to $\ell^2$.  For convenience, we will refer to the points on $\widetilde{\Gamma}_n$ as nodes \emph{at level $n$} in our tree.  When we speak of nodes \emph{above} $Q$, we mean nodes at the level immediately above that of $Q$, which are connected to $Q$ by an edge.

We may give an equivalent definition of the ramification tree in terms of function fields.
Identifying points of $C(\CC)$ with the corresponding valuations on $K_\CC$, points above $Q_0$, on the components of the curves $\widetilde{\Gamma}_n$,  correspond to valuations on the function fields of the corresponding components which
  extend $v$, and so we may take the valuations to be the nodes of our tree.  The valuation $v_{n+1}$ of $\CC(\widetilde{\Gamma}_{n+1}^{(i)})$ is linked
  to the valuation $v_n$ of $\CC(\widetilde{\Gamma}_{n}^{(j)})$ just if the former field extends the latter (i.e., $\widetilde{\Gamma}_{n+1}^{(i)}$ maps onto $\widetilde{\Gamma}_{n}^{(j)})$, and $v_{n+1}\mid v_n$.  Again, the weights on the edges are simply the ramification indices $e(v_{n+1}/v_n)$.  Note that the Galois orbits in $[\ell]^{-n}P$ correspond to the components of $\CC(\widetilde{\Gamma}_n)$, with $\CC(\widetilde{\Gamma}_n)\cong K_\CC(Q_n)$ for $Q_n\in [\ell]^{-n}P$ any representative of the appropriate Galois orbit.

It turns out that there are only three possible types of branching above a node in our tree: there might be
\begin{enumerate}
\item  $\ell^2$ edges above a given node, each
necessarily of weight 1; 
\item  $\ell$ edges of weight 1, and $\ell-1$ of weight $\ell$; or, 
\item $\ell$ edges, each of weight $\ell$. 
\end{enumerate}
The remainder of the section is devoted to proving this, and establishing the exact structure of the tree.


Let $\widehat{K_\CC}$ be the completion of $K_\CC$ with respect to $v$.  We recall some basic facts about extensions of local fields (see \cite[Section II.3]{serre_local}).  
  If $L=K_\CC(Q, E[\ell^n])$, for some $Q\in E(\overline{K_\CC})$ with $[\ell]^nQ=P$, then $L$ is a Galois extension of $K_\CC$.  If we fix a prolongation $w$
  of $v$ to $L$, then the decomposition group \[D_w=\left\{\sigma\in\Gal(L/K_\CC):w\circ\sigma= w\right\}.\] of $w/v$ is precisely the Galois group of
  $\widehat{L}/\widehat{K_\CC}$.  If $Q_1, ..., Q_g$ are a complete set of representatives of the $\Gal(L/K_\CC)$-orbits in $[\ell]^{-n}P$, then the
  prolongations of $v$ to the (distinct) fields $L(Q_i)$ are simply the valuations of the form $w\circ\sigma$, for $\sigma\in \Gal(L/K_\CC)$.  Two automorphisms generate the same valuation, if and only if they are in the same coset of $D_{w}$ in $\Gal(L/K_\CC)$.  Thus, the prolongations are exactly determined by the $D_{w}$-orbits (i.e., the $\Gal(\widehat{L}/\widehat{K_\CC})$-orbits) in $[\ell]^{-n}P$.  
Additionally, since all residual degrees are 1 (the residue field is always $\CC$), the ramification index $e(w/v)$ of the prolongation associated to the decomposition orbit containing $Q$ is exactly $[\widehat{K_\CC}(Q):\widehat{K_\CC}]$ (this is true because, as we will see below, $\widehat{K_\CC}(Q)/\widehat{K_\CC}$ is a Galois extension, even though $K_\CC(Q)/K_\CC$ may not be).
Thus, the nodes at level $n$ in our tree correspond to distinct orbits in $[\ell]^{-n}P$ under the decomposition group of some fixed valuation of $L$ extending $v$.  We will suppose throughout that we have extended $v$ in some way to $\overline{K_\CC}$, and the \emph{decomposition group} of a field $L$ will always refer to the decomposition group of the restriction of this valuation to $L$.

More generally, if $w$ is a prolongation of $v$ corresponding to the point $Q\in [\ell]^{-n}P$, then the nodes above $w$ in the ramification tree correspond to the decomposition orbits of points $Q'\in[\ell]^{-1}Q$.  Given a prolongation $w'$ corresponding to (the decomposition orbit of) $Q'$, the ramification index $e(w'/w)$ is simply \[[\widehat{K_\CC}(Q'):\widehat{K_\CC}(Q)]=[\widehat{K_\CC}(Q'):\widehat{K_\CC}]/[\widehat{K_\CC}(Q):\widehat{K_\CC}].\]


By Tate's $v$-adic uniformization, there is a unique $q\in \widehat{K_\CC}^*$ such that for any Galois extension $F/\widehat{K_\CC}$, we have an exact sequence
$$0\longrightarrow q^\ZZ\longrightarrow F^*\stackrel{\phi}{\longrightarrow} E(F)\longrightarrow  0$$
which respects the action of the Galois group (which acts trivially on $q$).  We point out that the units $R^*$ map, by $\phi$, onto the connected component $E_0(\widehat{K_\CC})$ containing the identity (see  \cite[p.~431]{jhs_advanced}).  It is also worth noting that $E[\ell^n]$ is generated by $\phi(q^{1/\ell^n})$ and $\phi(\zeta_{\ell^n})$, where $q^{1/\ell^n}$ is some $\ell^n$th root of $q$, and $\zeta_{\ell_n}$ is a primitive $\ell^n$th root of unity.  Note that, since $\zeta_{\ell^n}\in\CC\subseteq F$, the group $E(\widehat{K_\CC})$ contains at least cyclic $\ell^n$-torsion, for all $n$.

We now consider the points in $[\ell]^{-n}P$.  Recall that we are assuming $P\in E_0(\widehat{K_\CC})$, and thus we may fix, once and for all, a value $\beta\in R^*$ with $\phi(\beta)=P$.  By Lemma~\ref{powers}, $\beta\in (R^*)^{\ell^n}$, for all $n$, and so we will fix a compatible system of roots $\beta^{1/\ell^n}\in R^*$ (compatible in the sense that $(\beta^{1/\ell^{n+1}})^\ell=\beta^{1/\ell^n}$).  The elements of $[\ell]^{-n}P$ are precisely the images under $\phi$ of the points
\[\left\{\beta^{1/\ell^n}q^{a/\ell^n}\zeta_{\ell^n}^b: 0\leq a, b<\ell^n\right\}.\]
Suppose that $Q=\phi(\beta^{1/\ell^n}q^{a/\ell^n}\zeta_{\ell^n}^b)$.
Since $\zeta_{\ell^n}\in\CC\subseteq K_\CC$, for all $n$, and since $\beta^{1/\ell^n}\in R^*\subseteq \widehat{K_\CC}$, for all $n$, we note that
\[\widehat{K_\CC}(Q)=\widehat{K_\CC}(\beta^{1/\ell^n}q^{a/\ell^n}\zeta_{\ell^n}^b)=\widehat{K_\CC}(q^{a/\ell^n}).\]
In particular, if $q^a$ has order $\ell^m$ in $\widehat{K_\CC}^*/(\widehat{K_\CC}^*)^{\ell^n}$, then the conjugates of $Q$ by the decomposition group are simply the elements of the form $\phi(\beta^{1/\ell^n}q^{a/\ell^n}\zeta_{\ell^n}^b\zeta_{\ell^m}^c)$, for $c\in \ZZ/\ell^m\ZZ$.  In other words, the extension $\widehat{K_\CC}(Q)/\widehat{K_\CC}$ is a cyclic Galois extension  of order $\ell^{s}$, where $0\leq s\leq n-\ord_\ell(v(q))$ is the greatest value such that $a\equiv 0\MOD{\ell^s}$.

In particular, the quantity $a\in\ZZ/\ell^n\ZZ$ is an invariant of the decomposition orbit of $Q=\phi(\beta^{1/\ell^n}q^{a/\ell^n}\zeta^b_{\ell^n})$ (although there may be more than one orbit with the same value $a$), and hence an invariant of the corresponding node at level $n$ in the ramification tree.  Furthermore, if $a'\in\ZZ/\ell^{n+1}\ZZ$ is the corresponding quantity for a node corresponding to the decomposition orbit of $Q'\in[\ell]^{-1}Q$, then $a'\equiv a\MOD{\ell^n}$.   In order to describe the structure of the tree, we will set \[m=\ord_\ell(v(q))=\ord_\ell(v(j_E)),\] and say that a node has
%
\begin{enumerate}
\item Type A if $n< m$;
\item Type B$_r$ if $n\geq m$ and $a\equiv 0\MOD{\ell^{n-m}}$, and $0\leq r\leq m$ is the greatest value with $a\equiv 0\MOD{\ell^{n+r-m}}$; and 
\item Type C if $n\geq m$ and $a\not\equiv 0\MOD{\ell^{n-m}}$.
\end{enumerate}
The following lemma describes the structure of the ramification tree:
\begin{lemma}\label{ramificationtree}
The ramification tree observes the following rules:
\begin{enumerate}
\item all nodes at level $0$ through $m-1$ have Type A (if $m=0$, then there are no such nodes), and each of these nodes have $\ell^2$ distinct edges above them (necessarily each of weight 1);
\item all nodes at level $m$ are Type B$_r$ nodes, for some $0\leq r\leq m$, and at any level of the tree,
\begin{enumerate}
\item if $m=0$, then a Type B$_0$ node has $\ell$ Type B$_0$ nodes above it, with weight 1 each, and $\ell-1$ Type C nodes above it, with weight $\ell$ each;
\item if $m\geq 1$, then a Type B$_0$ node has $\ell$ Type C nodes above it, each with weight $\ell$;
\item a Type B$_r$ node, for $1\leq r<m$, has $\ell^2$ Type B$_{r-1}$ nodes above it, each with weight 1;
\item if $m\geq 1$, a Type B$_m$ node has $\ell$ nodes of Type B$_m$ above it, each of weight 1, and $\ell-1$ nodes of Type B$_{m-1}$ above it, each of weight $\ell$;
\end{enumerate}
\item each Type C node has $\ell$ Type C nodes above it, each with weight $\ell$.
\end{enumerate}
\end{lemma}

Thus, when $\ord_\ell(j_E)=0$, the ramification tree looks something like this:
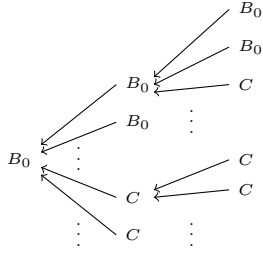
\begin{center}
\begin{tikzpicture}
	\draw (0,0) node [left] {\tiny $B_0$};

	\draw[<-] (0,0.2) -- (1, 1) node[right] {\tiny $B_0$};
	\draw[<-] (0,0.1) -- (1, 0.5) node[right] {\tiny $B_0$};
	\draw (0.5, 0.5) node[below] {\tiny $\vdots$};

	\draw[<-] (0,-0.1) -- (1, -0.5) node[right] {\tiny $C$};
	\draw[<-] (0,-0.2) -- (1, -1) node[right] {\tiny $C$};
	\draw (0.5, -0.5) node[below] {\tiny $\vdots$};

	\draw[<-] (1.5, 1.1) -- (2.5, 2) node[right] {\tiny $B_0$};	
	\draw[<-] (1.5, 1) -- (2.5, 1.5) node[right] {\tiny $B_0$};	
	\draw[<-] (1.5, 0.9) -- (2.5, 1) node[right] {\tiny $C$};	
	\draw (2, 1) node[below] {\tiny $\vdots$};

	\draw[<-] (1.5, -0.4) -- (2.5, 0) node[right] {\tiny $C$};	
	\draw[<-] (1.5, -0.5) -- (2.5, -0.4) node[right] {\tiny $C$};	
	\draw (2, -0.5) node[below] {\tiny $\vdots$};
\end{tikzpicture}
\end{center}
We should point out that the lemma above does not uniquely define the structure of the tree, but only gives the information that we will need in the proof of Theorem~\ref{main}.  In the proof below, however, it is pointed out that there are precisely $\ell^{m}\varphi(\ell^{m-r})$ nodes of Type B$_r$ at level $m$, for each $0\leq r\leq m$, where $\varphi$ is the Euler totient function. This fact (combined with Lemma~\ref{ramificationtree}) uniquely determines the full structure of the tree.

\begin{proof}
The claim that all nodes at levels $0$ through $m-1$ have Type A is clear from the definition, as is the claim that all nodes at level $m$ have Type B$_r$, for some $0\leq r\leq m$ (since the condition $a\equiv 0\MOD{\ell^0}$ is satisfied trivially).  We now conduct a case-by-case examination of the nodes immediately above nodes of any given Type, proving the various claims in the lemma.  Throughout, we take $Q=\phi(\beta^{1/\ell^n}q^{a/\ell^n}\zeta_{\ell^n}^b)$ as a representative of the appropriate decomposition orbit, and consider the decomposition orbits of possible \[Q'=\phi(\beta^{1/\ell^{n+1}}q^{a'/\ell^{n+1}}\zeta_{\ell^{n+1}}^{b'})\in[\ell]^{-1}Q.\]

\vspace{2mm}
\noindent\textbf{Case: Type A}
 
First, suppose that $n<m$.  Then $q^{a'/\ell^{n+1}}$ is $\widehat{K_\CC}$-rational, for all $a'\in \ZZ/\ell^{n+1}\ZZ$, since $q\in (\widehat{K_\CC}^*)^{\ell^m}$, and $n+1\leq m$.  In particular, every extension of the form $\widehat{K_\CC}(q^{a'/\ell^{n+1}})/\widehat{K_\CC}(q^{a/\ell^n})$ is trivial, and so there are $\ell^2$ distinct decomposition orbits above $Q$.  This gives $\ell^2$ distinct nodes above the node corresponding to $Q$, and each must be unramified, since their ramification indices sum to $\ell^2$ (or, simply because the inertia groups are trivial).  As noted above, if $n+1<m$, then these points are all Type A, while if $n+1=m$, then  the points above $Q$ have type B$_r$ for some $r$.  It is, in fact, quite easy to compute the number of nodes of Type B$_r$ at level $m$, for each $0\leq r\leq m$, thereby completely describing the structure of the tree.  The nodes of Type B$_r$ are in one-to-one correspondence with pairs $(a, b)$, with $b\in\ZZ/\ell^m\ZZ$, and $a\in\ZZ/\ell^m\ZZ$ of the form $\ell^ru$, for $u\in\ZZ/\ell^m\ZZ$ a unit.  In other words, there are precisely $\ell^{m}\varphi(\ell^{m-r})$ nodes of Type B$_r$ at level $m$, for each $0\leq r\leq m$.

\vspace{2mm}
\noindent\textbf{Case: Type B$_r$}

Now, suppose that $n\geq m$, and that $a\equiv 0\MOD{\ell^{n-m}}$, so that $Q$ has type B$_r$, for some $0\leq r\leq m$.  Since $q\in (\widehat{K_\CC}^*)^{\ell^m}$, say $q=q_0^{\ell^m}$, we see that the field
\[\widehat{K_\CC}(Q)=\widehat{K_\CC}(q^{a/\ell^n})=\widehat{K_\CC}(q_0^{a/\ell^{n-m}})\]
is simply $\widehat{K_\CC}$.  Now let $Q'=\phi(\beta^{1/\ell^{n+1}}q^{a'/\ell^{n+1}}\zeta_{\ell^{n+1}}^{b'})$ be an element of $\ell^{-1}Q$, so that $a'\equiv a\MOD{\ell^n}$, and $b'\equiv b\MOD{\ell^n}$. Consider the extension \[\widehat{K_\CC}(Q')/\widehat{K_\CC}(Q)=\widehat{K_\CC}(q^{a'/\ell^{n+1}})/\widehat{K_\CC}(q^{a/\ell^n}).\]
If $a'\equiv 0\MOD{\ell^{n+1-m}}$, then this extension is trivial, and each $Q'$ of this form is fixed by the decomposition group.  To each of these points (if there are any), corresponds a prolongation of $w$ to $K_\CC(Q')$ which is unramified.  If $a'\not\equiv 0\MOD{\ell^{n+1-m}}$, then this extension is cyclic of degree $\ell$.  The decomposition orbit of $Q'$, in this case, corresponds to a prolongation of $w$ to $K_\CC(Q')$ which is ramified of index $\ell$.

Now, if $Q$ has Type B$_r$, for some $r\geq 1$, then we have that $m\geq 1$, and $a\equiv 0\MOD{\ell^{n+1-m}}$.  So for all points $Q'$ above $Q$, we have $a'\equiv a\equiv 0\MOD{\ell^{n+1-m}}$, since $a'\equiv a\MOD{\ell^n}$, and $n+1-m\leq n$.  In this case, all points $Q'$ will be fixed by the decomposition group, and so correspond to unramified prolongations of $w$ to $K_\CC(Q')$.  If $r<m$, we have $a'\equiv a\not\equiv 0\MOD{\ell^{n+r+1-m}}$, and so each $Q'\in\ell^{-1}Q$ is a point of Type B$_{r-1}$.  If, on the other hand, $r=m$, then we have $a\equiv 0\MOD{\ell^n}$.  There is one choice $a'\in\ZZ/\ell^{n+1}\ZZ$ with $a'\equiv a\MOD{\ell^{n}}$ and $a'\equiv 0\MOD{\ell^{n+1}}$, and $\ell-1$ choices with $a'\not\equiv 0\MOD{\ell^{n+1}}$.  By the same argument, the choice $a'=0$ yields $\ell$ decomposition-invariant points, and so $\ell$ distinct, unramified prolongations of $w$.  Each of those points will again have Type B$_m$.  The $\ell-1$ choices with $a'\neq 0$ each yield a single decomposition orbit, and so a single prolongation of $w$ with ramification index $\ell$.  Since $a'\equiv 0\MOD{\ell^n}$, but $a'\not\equiv 0\MOD{\ell^n}$, these $\ell-1$ points have type B$_{m-1}$.
 
 Now suppose that $Q$ has Type B$_0$, which we will subdivide into two cases, depending on whether or not $m=0$.  If $m\geq 1$, then $a\equiv 0\MOD{\ell^{n-m}}$, but $a\not\equiv 0\MOD{\ell^{n+1-m}}$.  It follows that any point $Q'\in \ell^{-1}Q$, which must satisfy $a'\equiv a\MOD{\ell^n}$, satisfies $a'\not\equiv 0\MOD{\ell^{n-m}}$.  These points, then, are all Type C.  For each of these values of $a'\in\ZZ/\ell^{n+1}/\ZZ$, we find that $Q'$ is in a decomposition orbit of size $\ell$.  Thus, there are $\ell$ prolongations of $w$, each with ramification index $\ell$.  On the other hand, if $m=0$, then the condition on $a$ is simply that $a\equiv 0\MOD{\ell^n}$.  Consider the different values $a'\in\ZZ/\ell^{n+1}\ZZ$ with $a'\equiv a\MOD{\ell^n}$.  If $a'=0$, then the decomposition group acts trivially on $Q'$.  This gives $\ell$ distinct prolongations of $w$, each with ramification index $\ell$.  Each of the points corresponding to these prolongations has Type B$_0$.  The $\ell-1$ choices of $a'\neq 0$ each yield one prolongation, with Type $C$.
 
\vspace{2mm}
 \noindent\textbf{Case: Type C}
 
 In this case, we have $a\not\equiv 0\MOD{\ell^{n-m}}$, and so if $a'\in\ZZ/\ell^{n+1}\ZZ$, with $a'\equiv a\MOD{\ell^n}$, we have $a'\not\equiv 0\MOD{\ell^{n+1-m}}$.  Thus, any point $Q'\in\ell^{-1}Q$ is a Type C point.  Now, any of the extensions $\widehat{K_\CC}(q^{a'/\ell^{n+1}})/\widehat{K_\CC}(q^{a/\ell^{n}})$ must be cyclic of order $\ell$, since $q^{a'}$ generates a cyclic subgroup of $\widehat{K_\CC}^*/(\widehat{K_\CC}^*)^{\ell^{n+1}}$ of order $\ell^{n+1-m}$, while $q^a$ generates a cyclic subgroup of $\widehat{K_\CC}^*/(\widehat{K_\CC}^*)^{\ell^n}$ of order $\ell^n$.  Thus, there are $\ell$ primes above $w$, each with ramification index $\ell$, and corresponding to a point of Type C.
\end{proof}


\section{Proof of Theorem~\ref{main}}\label{mainproof}

We proceed with the proof of the main result, first making several reductions in order to simplify the argument.  For simplicity, we introduce the following notation:
\[X(k, B_1, \Es, P, t)=\Big\{Q\in \Es_t(\overline{k}):[k(Q):k]\leq B_1\text{ and }\ell^nQ=P_t\text{ for some }n\geq 1\Big\}.\]
Our first lemma shows that we are free to replace $\Es\rightarrow C$ by some base extension defined over an algebraic extension of $k$.  In other words, thinking in terms of the generic fibre $E/K$, it suffices to prove Theorem~\ref{main} for $E$ over an algebraic extension $K'/K$.
\begin{lemma}\label{base}
Let $k'/k$ be an algebraic extension, and let $\phi:C'\rightarrow C$ be a dominant morphism of curves defined over $k'$.  If Theorem~\ref{main} is true for an elliptic surface $\Es'\rightarrow C'$ birational to $\Es\times_C C'$, then it is true for $\Es$.
\end{lemma}

\begin{proof}
In proving Theorem~\ref{main}, it is clear that we are proving something stronger if we pass to an algebraic extension of $k$.  So we will, replacing $k$ by $k'$ if necessary, suppose that $\phi:C'\rightarrow C$ is defined over $k$.  Now let $\Es'$ be birational to $\Es\times_C C'$, and let $P:C\rightarrow\Es$ be a section.  Then $P$ lifts uniquely to a section $P':C'\rightarrow\Es'$, and since $\Es_t'\cong \Es_{\phi(t)}$ over $k$, for all but finitely many $t\in C'(\overline{k})$, we have immediately that
\[X(k, B_1, \Es, P, \phi(t))= X(k, B_1, \Es', P', t).\]
Now, for any $s\in C(\overline{k})$, there is some $t\in C'(\overline{k})$ with $\phi(t)=s$, and we have $[k(t):k]\leq \deg(\phi)[k(s):k]$.  Thus, we may conclude Theorem~\ref{main} for $\Es$ and $P$, with $B_1=D_1$ and $B_2=D_2$ by applying Theorem~\ref{main} to $\Es'$ and $P'$, with $B_1=D_1$ and $B_2=\deg(\phi)D_2$.
\end{proof}

From this point forward, we will assume that $\Es$ has been replaced with an appropriate base extension $\Es'$, so that there is at least one place $v\in C(\overline{k})$ such that the fibre $\Es_v$ is multiplicative.

\begin{lemma}\label{multbyN}
Let $k$, $\Es$, and $\ell$ be as in Theorem~\ref{main}, and let $N\geq 1$.  If Theorem~\ref{main} is true for $P=NP'$, then it is true for $P=P'$.
\end{lemma}

\begin{proof}
Suppose there is some point $Q\in\Es_t(\overline{k})$ with $\ell^nQ=P'_t$.  Then $\ell^n NQ=NP'_t$, and clearly $k(NQ)\subseteq k(Q)$ (as multiplication-by-$N$ is given by rational maps defined over $k$).  In particular, in the notation above, the image of $X(k, B_1, \Es, P', t)$ by the morphism $[N]:\Es_t\rightarrow \Es_t$ is wholly contained in the set $X(k, B_1, \Es, NP', t)$.
But this means that the cardinality of the former set  is at most $N^2$ times the cardinality of the latter.
\end{proof}

\begin{lemma}\label{merel}
Theorem~\ref{main} is true for $P$ a point of finite order on $E(K)$.
\end{lemma}

\begin{proof}
In light of Lemma~\ref{multbyN}, we may replace $P$ by $NP$ for any $N\geq 1$, and so it suffices to consider the case where $P=\Ocal$ is the identity section.  But 
\[\Big\{Q\in \Es_t(\overline{k}):\ell^nQ=\Ocal_t\text{ for some }n\geq 1\Big\}\]
is  just $\Es_t[\ell^\infty]\setminus\{\Ocal_t\}$, and so in particular is a subset of $\Es_t(\overline{k})_\mathrm{Tors}$.  

A simple modification of a theorem of Merel (see \cite[Corollary~6.64]{jhs_ads} for the appropriate variant of Merel's Theorem) implies that for all elliptic curves $E/F$, where $F/\QQ$ is a number field, the set
\[\bigcup_{[L:F]\leq n}E(L)_\mathrm{Tors}\]
is bounded in cardinality by some quantity that depends only on $n$ and $[F:\QQ]$.  Since the (smooth) fibres $\Es_t$, for $[k(t):k]\leq B_2$, are all elliptic curves defined over number fields of degree at most $B_2[k:\QQ]$, invoking the claim above with $n=B_1B_2[k:\QQ]$ gives us a uniform bound on the sizes of the sets $X(k, B_1, \Es, \Ocal, t)$, for $t\in C(\overline{k})$ with $[k(t):k]\leq B_2$.
\end{proof}

We now prove Lemma~\ref{preim}, stated in the introduction, relying heavily on Lemma~4.5 of \cite{preimpaper}.

\begin{proof}[Proof of Lemma~\ref{preim}]
For any morphism of curves $\phi:X\rightarrow Y$ with ramification divisor $R_\phi$, we set
\[\rho(\phi)=\frac{\deg R_\phi}{2\deg\phi},\]
and
recall that we have assumed that there exist constants $c_1>0$ and $c_2$ such that
\[\rho(\phi_n)\geq c_12^n-c_2.\]
Lemma~4.5 of \cite{preimpaper} states that if 
\[X_0\stackrel{\phi_1}{\longleftarrow} X_1\stackrel{\phi_2}{\longleftarrow}\cdots \stackrel{\phi_N}{\longleftarrow} X_N\]
is a tower of (smooth projective) curves, equipped with non-constant morphisms, all defined over  a number field $k$, then, setting
\begin{gather*}
B_N=\min_{1\leq m\leq N} 2^{N-m}\rho(\phi_m)\\
b_N=\min_{1\leq m\leq N} \rho(\phi_m),
\end{gather*}
the set
\[\{P\in X_N(\overline{k}):[k(P):k]<B_N\text{ and }[k(\phi_1\circ\cdots\circ\phi_N(P)):k]\geq b_N\}\]
is finite.  Note that if any of the maps $\phi_n$ is unramified, the result is trivial, since in this case $B_N=b_N=0$.  Thus, we fix $n_0$ such that for $n> n_0$, $\rho(\phi_n)>\frac{1}{2}c_12^n$ (for example, we could take $n_0$ to be anything larger than $\log_2(2c_2/c_1)$, assuming $c_2>0$, or $n_0=0$ otherwise).  We then apply the lemma to the tower of curves
\[\PP^1\stackrel{\psi_1}{\longleftarrow} \PP^1\stackrel{\psi_2}{\longleftarrow} C_{n_0}\stackrel{\psi_3=\phi_{n_0+1}}{\longleftarrow}C_{n_0+1}\stackrel{\psi_4=\phi_{n_0+2}}{\longleftarrow}\cdots,\]
where the two leftmost maps are any morphisms of degree at least 2 (defined over $k$).  On the one hand, the conditions on $\phi_n$, and the fact that $\psi_1$ and $\psi_2$ are not unramified, ensure that $\rho(\psi_n)\geq \epsilon 2^n$, for some $\epsilon>0$.  Thus, for any $N\geq 1$,
\[B_N=\min_{1\leq m\leq N} 2^{N-m}\rho(\psi_m)\geq \min_{1\leq m\leq N} 2^{N-m}\epsilon 2^m=\epsilon2^N.\]
On the other hand, $\psi_1:\PP^1\rightarrow\PP^1$, and so the Hurwitz formula guarantees that
\[b_N\leq\rho(\psi_1)=\frac{2d-2}{2d}<1.\]
By the lemma from \cite{preimpaper}, we know that for any $N>n_0$, the set
\[\{P\in C_N(\overline{k}):[k(P):k]<\delta2^{N}\text{ and }[k(\psi_1\circ\cdots\circ\psi_N(P)):k]\geq 1\}\]
is finite, with $\delta=\epsilon 2^{2-n_0}$.  Since the condition $[k(\psi_1\circ\cdots\circ\psi_N(P)):k]\geq 1$ is trivial, then, we have proven the lemma (with the more explicit information that it suffices to take $N$ larger than $\log_2(B/\delta)$).
\end{proof}

We now proceed with the body of the proof of the main result.  Let $k$, $\Es$, $P$, $\ell$, $B_1$, and $B_2$ be as in the statement of Theorem~\ref{main}.  In light of Lemma~\ref{merel}, we will suppose that $P$ is a section of infinite order.  The $j$-invariant $j_\Es:C\rightarrow\PP^1$ was assumed non-constant, and hence is dominant.  In particular, over $\CC$, it must have a pole.  If necessary, we may replace $\Es$ with an elliptic surface $\Es'$, birational to a base extension of $\Es$, such that this pole is a place of multiplicative reduction for $\Es'$.  After replacing $\Es$ with this base extension (employing Lemma~\ref{base}), the generic fibre $E$  of $\Es$, over $K=k(C)$, has split multiplicative reduction at $v$, over $K_\CC=K\otimes_k\CC$.  Finally, in light of Lemma~\ref{multbyN}, we will also replace  $P$ by $v(j_E)P$,  in order to ensure that $P\in E_0(\widehat{K_\CC})$, where $\widehat{K_\CC}$ is the completion of $K_\CC$ at $v$.

Let $\Gamma_0\subseteq \Es$ be the image of $P$ and, as in the introduction, let $\Gamma_{n+1}=[\ell]^{-1}\Gamma_n$.  If $t\in C(\overline{k})$ has $[k(t):k]\leq B_1$, and $Q\in \Es_t(\overline{k})$ has $[k(Q):k]\leq B_2$, and $\ell^n Q=P_t$, then the pair $(Q, t)$ corresponds to a point on $\Gamma_n$ of degree at most $B=B_1B_2$ over $k$, and hence a point of degree at most $B$ on the normalization $\widetilde{\Gamma}_n$.  We will use Lemma~\ref{preim} to bound $n$.

First of all, note that since the tree of components of the curves $\widetilde{\Gamma}_n$ contains only finitely many paths (by Lemma~\ref{orbits}), it suffices to prove the result for any of the finitely many distinct towers
\[C_0\stackrel{\phi_0}{\longleftarrow}C_1\stackrel{\phi_1}{\longleftarrow}\cdots,\]
where $C_n\subseteq \widetilde{\Gamma}_n$ is irreducible, and $\phi_n$ is the map induced by $[\ell]:\Es\rightarrow\Es$.  In particular, we may choose $n_0$ large enough that for all $n\geq n_0$, $C_{n+1}$ is the \emph{only} component of $\widetilde{\Gamma}_{n+1}$ which maps, by the map $\widetilde{\Gamma}_{n+1}\rightarrow\widetilde{\Gamma}_n$ induced by $[\ell]:\Es\rightarrow\Es$, to $C_n$.  Now suppose that $w$ is a place above $v$ corresponding to a point on $C_{n_0}$.  We claim that there is a place $w'\mid w$, corresponding to a point of a curve of $C_n$, for some $n\geq n_0$, which has Type C (in the sense of Section~\ref{tate}).  Of course, if $w$ itself has Type C, then we're done.  Suppose that the $w$ is a node in the ramification tree of Type B$_r$, for some $0\leq r\leq m$.  By Lemma~\ref{ramificationtree}, then, there is a node of Type C above $w$ if $w$ has Type B$_0$, and a node of Type B$_{r-1}$ above $w$ otherwise.  By induction, then, we eventually have a node $w'\mid w$ of Type C in the ramification tree.  Finally, if $w$ has Type A, then there is some node of Type B$_0$ above $w$, and we apply the previous case.

So, increasing $n_0$ if necessary, and replacing $w$ with the node of Type C above it, we may simply assume that $w$ has Type C.  Now, every Type C node in the ramification tree splits into $\ell$ Type C nodes with ramification index $\ell$, at the next level of the tree.  Thus, the curve $C_{n_0+m}$ contains $\ell^m$ distinct points above $w$, all of Type C.   In other words, the ramification divisor of the morphism
\[\phi_{n_0+m}:C_{n_0+m}\rightarrow C_{n_0+m+1}\]
has degree at least $\ell^m(\ell-1)$.  Since each of the maps $\phi_n$ has degree at most $\ell^2$, it follows that for $n> n_0$,
\[\rho(\phi_n)\geq\frac{\ell^{n-n_0}(\ell-1)}{\ell^2}\geq \ell^n\left(\frac{\ell-1}{\ell^{n_0+2}}\right).\]
Since there are only finitely many $n\leq n_0$, this shows that there are constants $c_1>0$ and $c_2$ such that
\[\rho(\phi_n)\geq c_12^n-c_2\]
(in fact, we could replace the $2^n$ with $\ell^n$, but this provides no gains after the application of Lemma~\ref{preim}).  We are now in a position to apply Lemma~\ref{preim}.  The lemma tells us that for any $B\geq 1$, there is an $N(B)$ such that $C_{N(B)}(\overline{k})$ contains only finitely many points $z$ with $[k(z):k]\leq B$.  Applying the argument to each of the finitely many towers of components of the curves $\widetilde{\Gamma}_n$ proves the same thing for those curves.

 Given $B\geq 1$, let $Y_B\subseteq C(\overline{k})$ be the (finite) set of $t$ corresponding to points in $\Gamma_{N(B)}(\overline{k})$ of degree at most $B$.
Now suppose that $t\in C(\overline{k})$, and $Q\in\Es_t(\overline{k})$, with $\ell^n Q=P_t$ and $[k(t, Q):k]\leq B$.  If $t\not\in Y_B$, then $n\leq n_0$.  The number of points $Q$ is at most
\begin{eqnarray*}
\#\ell^{-1}P_t+\#\ell^{-2}P_t+\cdots+\#\ell^{-n_0}P&=&\ell^2+\ell^4+\cdots+\ell^{2n_0(B)}\\
&=&\frac{\ell^{2n_0(B)+1}-\ell^2}{\ell^2-1},
\end{eqnarray*}
which does not depend on $t$.

Now suppose that $t\in Y_B$, and that $\Es_t$ is non-singular.  If $\hat{h}_t:\Es_t(\overline{k})\rightarrow\RR^+$ is the N\'{e}ron-Tate height on the elliptic curve $\Es_t$, then for any $Q\in \Es_t(\overline{k})$ with $\ell^nQ=P_t$, for some $n\geq 1$, we have \[\h_t(Q)=\ell^{-2n}\h_t(P_t)\leq \h_t(P_t).\]  Thus, the number of points in $\Es_t(\overline{k})$ with $\ell^n Q=P_t$ for some $n\geq 1$, and $[k(Q):k]\leq B_1$ is finite.  So, since $Y_B$ is finite, we have an upper bound on the size of the set
\[\Big\{Q\in \Es_t(\overline{k}):[k(Q):k]\leq B_1\text{ and }\ell^nQ=P_t\text{ for some }n\geq 1\Big\},\]
for $t\in C(\overline{k})$ with $[k(t):k]\leq B_2$, whether $t\in Y_B$ or not.  This proves Theorem~\ref{main}.


\section{The proof of Theorems~\ref{special_th} and \ref{torsion}}\label{sharp}

As mentioned in the introduction, Theorem~\ref{main} cannot be particularly improved, since we are always free to replace $P$ be $\ell^NP$, for some $N$, thereby arbitrarily  increasing the number of rational points in $\bigcup [\ell]^{-n}P_t$, on each fibre.

Requiring that $\ell$ be non-special, we can prove the stronger claim of Theorem~\ref{special_th}, namely that if $P$ is not of the form $\ell P_0$, for any section $P_0:C\rightarrow \Es$, then
\[\#\Big\{Q\in \Es_t(k):\ell^n Q=P_t\text{ for some }n\geq 1\Big\}\leq \ell^2,\]
for all but finitely many places $t\in C(k)$, provided that $\Es$ has at least one multiplicative fibre.  Taking a cue from work of Baragar and McKinnon \cite{mckinnon}, we note that we may replace the upper bound with 0  if $\Es$ has  at least 5 distinct multiplicative fibres (4 multiplicative fibres suffice if $\ell=5$, or 3  if $\ell\geq 7$).  We should note that, since the group of sections on $\Es$ is finitely generated, $P$ is an $\ell$th multiple of another section only for finitely many primes $\ell$. 

The result follows from the Mordell Conjecture (now a theorem of Faltings), once one shows that the curve $\widetilde{\Gamma}_2$ (or, with the additional hypotheses, $\widetilde{\Gamma}_1$) is irreducible, and has genus at least 2.  For if this is the case, then there are only finitely many fibres on which $[\ell]^{-2}P_t$ (respectively, $[\ell]^{-1}P_t$) contains any $k$-rational points at all.  The result follows, since $[\ell]^{-1}P_t$ contains at most $\ell^2$ points.  Thus, Theorem~\ref{special_th} is proven once we establish:

\begin{lemma}
Let $\ell$ be a non-special prime for $\Es$, and let $P:C\rightarrow\Es$ be a section which is \emph{not} an $\ell$th multiple, and suppose that $j_\Es$ is non-constant.  Then $\widetilde{\Gamma}_2$ is an irreducible curve of genus at least 4.
If we suppose, additionally, that $j_\Es:C\rightarrow\PP^1$ has at least 5 distinct poles over $\CC$ (at least 4 poles if $\ell=5$, or at least 3 poles if $\ell\geq 7$), then  $\widetilde{\Gamma}_1$ is an  irreducible curve of genus at least 2.
\end{lemma}

\begin{proof}
The irreducibility follows from the results in Section~\ref{orbits}.  In particular, since $\ell$ is not a special prime, Lemma~\ref{open_image} tells us that Galois group of the covering $\widetilde{\Gamma}_n\rightarrow\Gamma_0$ is isomorphic (in the natural way) to a semi-direct product $\left(\ZZ/\ell^n\ZZ\right)^2\rtimes\SL_2(\ZZ/\ell^n\ZZ)$.  In particular, the action is transitive, and the curve $\widetilde{\Gamma}_n$ is irreducible.

For convenience, we will work in the function field setting, considering the generic fibre $E/K_\CC$.  We will denote the function fields (over $\CC$) of $\widetilde{\Gamma}_1$ and $\widetilde{\Gamma}_2$ by $F_1$ and $F_2$, respectively.

Since $j_\Es:C\rightarrow\PP^1$ is not constant, it is dominant, and so must have a pole.  Let $v$ be a pole of $j_\Es:C\rightarrow\PP^1$, and suppose for the time being that  $E$ has multiplicative reduction at $v$.
If we have $P\in E_0(\widehat{K_\CC})$ then, by Lemma~\ref{ramificationtree} (since $v(j_\Es)$ is prime to $\ell$), the place $v$ is a point of Type B$_0$, in the terminology of Section~\ref{tate}.  A quick examination of the structure of the tree (referring to Lemma~\ref{ramificationtree}) shows that there are $\ell$ places of Type B$_0$ of $F_1$ above $v$, and above each of these there are $\ell-1$ places $w$ of $F_2$ with $e_w(F_2/F_1)=\ell$.  Similarly, there are $\ell-1$ places of Type C for $F_1$ above $v$, and above each of these, $\ell$ places $w$ of $F_2$  with $e_w(F_2/F_1)=\ell$.  Thus the map (of degree $\ell^2$)  $\widetilde{\Gamma}_2\rightarrow\widetilde{\Gamma}_1$ has ramification divisor of degree at least $2\ell(\ell-1)^2$.  By the Riemann-Hurwitz formula,
\[2\left(g(\widetilde{\Gamma}_2)-1\right)\geq\ell^22\left(g(\widetilde{\Gamma}_1)-1\right)+2\ell(\ell-1)^2,\]
or, using the trivial bound $g(\widetilde{\Gamma}_1)\geq 0$,
\[g(\widetilde{\Gamma}_2)\geq\ell^3-3\ell^2+\ell+1\geq 4\]
(recall that $\ell\geq 3$).

  If $P\not\in E_0(\widehat{K_\CC})$, then we cannot apply Lemma~\ref{ramificationtree}.  However, the general approach of Section~\ref{tate} still applies.  If
\[0\longrightarrow q^{\ZZ}\longrightarrow \widehat{K_\CC}\stackrel{\phi}{\longrightarrow} E(\widehat{K_\CC})\longrightarrow 0\]
is the Tate uniformization of $E$ at $v$, as in Section~\ref{tate}, then set $m=v(q)=-v(j_\Es)$, and write $q=q_0^m$, for some $q_0\in\widehat{K_\CC}^*$.  Since $q_0$   is a uniformizer for $v$, we may write  $P=\phi(u q_0^p)$, for some $0< p< m$ and some $v$-unit $u$ (we may take $0<p<m$ because $q^\ZZ=\ker(\phi)$).  The places of $F_2$ above $v$ correspond to decomposition orbits (relative to a fixed prolongation of $v$) of points of the form
\[Q'=\phi(u^{1/\ell^2}q_0^{(p+am)/\ell^2}\zeta_{1/\ell^2}^b),\]
for $a$ and $b\in\ZZ/\ell^2\ZZ$.  By hypothesis, $m$ is prime to $\ell$, and so the function $q\mapsto p+am$ simply permutes $\ZZ/\ell^2\ZZ$.  In other words, the places of $F_2$ above $v$ simply correspond to the decomposition orbits of points of the form
\[Q'=\phi(u^{1/\ell^2}q_0^{a/\ell^2}\zeta_{1/\ell^2}^b),\]
for $a$ and $b\in\ZZ/\ell^2\ZZ$.  Exactly as in Section~\ref{tate}, the $\ell(\ell-1)$ choices of $a$ such that $a\not\equiv 0\MOD{\ell^2}$ each yield a place $w$ of $F_2$ for which  $e_w(F_2/F_1)=\ell$.  The $\ell-1$ choices of $a\equiv 0\MOD{\ell}$ but $a\not\equiv 0\MOD{\ell^2}$ give $\ell(\ell-1)$ places $w$ of $F_2$ for which $e_w(F_2/F_1)=\ell$.  Just as in the previous case, we obtain
\[g(\widetilde{\Gamma}_2)\geq\ell^3-3\ell^2+\ell+1\geq 4.\]

Now we suppose that $v$ is a pole of $j_\Es:C\rightarrow\PP^1$, but that $v$ is a place of additive reduction. Then  there is some quadratic extension $K'/K_\CC$, and an elliptic curve $E'/K'$ which is $K'$-isomorphic to $E$, such that $E'$ has multiplicative reduction at $v$  (see \cite[p.~442]{jhs_advanced}).  Moreover,  $v=(v')^2$ ramifies in this extension.  Let $Q_1, Q_2\in E'(\overline{K'})$ with $\ell^2Q_2=\ell Q_1=P$.  We have shown that there are $\ell$ places $w$ of $K'(Q_1)$ above $v'$ with $e_w(K'(Q_1)/K')=1$, and $\ell-1$ with $e_w(K'(Q_1)/K')=\ell$.  In other words, there are $\ell$ places $w$ of $K'(Q_1)$ above $v$ with $e_w(K'(Q_1)/K)=2$, and $\ell-1$ with $e_w(K'(Q_1)/K)=2\ell$.  But $K'(Q_1)=K'F_1$, and so a prime of $K'(Q_1)$ is totally ramified in the extension $K'/K$ only if it is in the extension $K'(Q_1)/F_1$.  In particular, every prime $w$ of $K'(Q_1)$ above $v$ satisfies $e_w(K'(Q_1)/F_1)=2$.  Thus each place $w$ of $F_1$ above $v$ extends uniquely to a place $w'$ of $K'(Q_1)$, and we have $e_w(F_1/K)=\frac{1}{2}e_{w'}(K'(Q_1)/K)$.  It follows that there are $\ell$ place $w$ of $F_1$ with $e_w(F_1/K)$, and $\ell-1$ with $e_w(F_1/K)=\ell$.  After conducting the same analysis for the extension $F_2/F_1$, we see that the factorizations of $v$ in these extensions are identical to the previous case.  The Hurwitz formula again gives $g(\widetilde{\Gamma}_2)\geq 4$.

Now we consider the genus of $\widetilde{\Gamma}_1$, in terms of the number of poles of $j_\Es$.
By the arguments above, if $v$ is a pole of $j_\Es$, then there are at least $\ell-1$ places of $\widetilde{\Gamma}_1$  above $v$ at which the map $\widetilde{\Gamma}_1\rightarrow \widetilde{\Gamma}_0$ is ramified with index $\ell$.  Thus, the ramification divisor of the map $\widetilde{\Gamma}_1\rightarrow \widetilde{\Gamma}_0$ has degree at least $N(\ell-1)^2$, where $N$ is the number of distinct poles of $j_\Es$ (over $\CC$).  By The Riemann-Hurwitz formula (since the morphism $\widetilde{\Gamma}_1\rightarrow\widetilde{\Gamma}_0$ has degree $\ell^2$), we have
\[g(\widetilde{\Gamma}_1)\geq 1-\ell^2+\frac{1}{2}N(\ell-1)^2.\]
This is of no use to us if $N\leq 2$, but one can check that if $N\geq 5$, then $g(\ell)\geq 2$ for all $\ell\geq 3$.  Furthermore, if $N\geq 4$, then $g(\ell)\geq 8$ for $\ell\geq 5$, and if $N\geq 3$, then $g(\ell)\geq 6$ for all $\ell\geq 7$.

If $\Gamma_0\cong C$ has genus at least 1, then the estimate on the ramification of the map $\widetilde{\Gamma}_1\rightarrow\widetilde{\Gamma}_0$, and the Hurwitz formula, gives
\[g(\widetilde{\Gamma}_1)\geq 1+\frac{1}{2}(\ell-1)^2\geq 3\]
for $\ell\geq 2$, regardless of the number of poles of $j_\Es$ (provided that $j_\Es$ is non-constant).
\end{proof}

The remark after the statement of Theorem~\ref{special_th} is proved by a similar argument.  In particular, in this simplified case, one know that all of the curves $\widetilde{\Gamma}_n$ are irreducible, and an examination of the ramification tree shows that the map
\[\phi_{n+1}:\widetilde{\Gamma}_{n+1}\rightarrow\widetilde{\Gamma}_{n}\]
has ramification of degree at least $(n+1)\ell^n(\ell-1)^2$.   Lemma~4.5 of \cite{preimpaper} now implies the an upper bound on $n$ such that $[\ell]^{-n}P_t$ contains points of degree at most $D$ on infinitely many fibres. The lower bound comes from observing that the points in $[\ell]^{-n}P_t$ each have degree at most $\ell^{2n}$ over $k$.

\begin{proof}[Proof of Theorem~\ref{torsion}]
Let $\Es$ and $S$ be as in the statement of the theorem, and suppose that $\Es_t(k)$ contains a point of order $N$, for $N$ not an $S$-unit.  Then $\Es_t(k)$ contains a point of order $\ell$, for some prime $\ell\not\in S$.  Note that, by Merel's Theorem, $\ell$ is bounded in terms of $[k:\QQ]$.  Thus, it suffices to show that the set of $t\in C(k)$ such that $\Es_t(k)$ contains a point of order $\ell$, for any given $\ell\not\in S$, is finite.


Now, fix $\ell\not\in S$.  Since $E[\ell]\setminus\{\Ocal\}$ is Galois-irreducible, the curve $[\ell]^{-1}\Ocal$ has two components, one of which is birational to $C$ (this is the curve on $\Es$ corresponding to $\Ocal$).  Let $C'$ be the normalization of the component birational to $C$, and let $\Gamma$ be the normalization of the other component.  If $j_\Es$ has a pole at the place $v$ on $C$, then (by the same argument as in the proof of Theorem~\ref{special_th}) there are a total of $\ell$ places above $v$, on the union of these two curves, at which the map induced by $[\ell]$ is unramified, and $\ell-1$ at which the map has ramification index $\ell$.  Since $C'\cong C$, all of the ramified places must lie on $\Gamma$, and so the places above $v$ contribute $(\ell-1)^2$ to the degree of the ramification divisor of the map $\Gamma\rightarrow C$ (this map has degree $\ell^2-1$).

If $C$ has genus 1 (or greater), then the existence of a single pole of $j_\Es$ gives the lower bound
\[g(\Gamma)\geq1+\frac{1}{2}(\ell-1)^2\geq 3,\]
for $\ell\geq 3$.  Otherwise, the existence of $N$ distinct poles of $j_\Es:C\rightarrow\PP^1$ gives a lower bound of
\[g(\Gamma)\geq 1-(\ell^2-1)+\frac{N}{2}(\ell-1)^2\geq 3\]
for $N\geq 5$ and $\ell\geq 3$.
\end{proof}

It is worth remarking that, in the notation of the proof above, there is an obvious map $\Gamma\rightarrow X_1(\ell)$, where $X_1(\ell)$ is the usual modular curve.  Thus, we always have $g(\Gamma)\geq g(X_1(\ell))\geq 2$ for $\ell\geq 13$.


\section{The proof of Theorem~\ref{cheb}}\label{chebproof}

Finally, we prove Theorem~\ref{cheb}.  Let $E/k$ be an elliptic curve defined over a number field, and fix a rational prime $\ell$ and a value $M\geq 0$.  Let $n\geq 1$ be the least integer such that
$M<
\ell^{2n}$,
and let $F=k(E[\ell^n], [\ell]^{-n}P)$.  Then $F/k$ is a Galois extension, and if $\pf$ is a prime of $k$
 whose Frobenius element in $\Gal(F/k)$ is trivial, then $E_\pf(k_\pf)$ contains  $E_\pf[\ell^n]$, and contains a point $Q$ such that $\ell^nQ=P_\pf$.  In particular, for each $m\leq n$,  there are at least $\ell^{2m}$
values $R\in E_\pf(k_\pf)$ such that $\ell^m R=P_\pf$ for some $m\geq 1$. 
Thus, there are at least
\[\ell^{2n}+\ell^{2(n-1)}+\cdots+\ell^2=\frac{\ell^{2(n+1)}-1}{\ell^2-1}-1\]
points $R\in E_\pf(k_\pf)$ such that $\ell^mR=P_\pf$, for some $m\geq 1$.
  On the other hand, the density of this set of primes (removing the finitely many bad primes) is at least $[F:k]^{-1}$, by the Chebotarev density theorem.  Since $\Gal(F/k)$ embeds in $E[\ell^n]\rtimes\GL_2(\ZZ/\ell^n\ZZ)$, we have
$$[F:k]\leq \#\left(E[\ell^n]\rtimes\GL_2(\ZZ/\ell^n\ZZ) \right)\leq \ell^{2n}\times \ell^{3n-2}(\ell^2-1)\leq \ell^{5n}.$$
Thus, since $M\geq\ell^{2n-2}$,  our chosen set of primes has density at least
\[M^{-\frac{2}{5}(1-\frac{1}{n})}\geq M^{-\frac{2}{5}(1-\frac{2\log\ell}{\log M})}.\]

It is worth noting that, since an analogue of the Chebotarev density theorem is true for function fields in positive characteristic \cite{murty}, the same is true of Theorem~\ref{cheb}.

\end{document}